\documentclass[11pt]{amsart}
\usepackage{multirow,bigdelim}
\usepackage{amsfonts}
\usepackage{dsfont}
\usepackage{amsmath}
\usepackage{cases}
\usepackage{mathrsfs}
\usepackage{hyperref}
\usepackage{multimedia}
\usepackage{graphicx}
\usepackage{indentfirst}
\usepackage{bm}
\usepackage{amsthm}
\usepackage{mathrsfs}
\usepackage{latexsym}
\usepackage{amsmath}
\usepackage{amssymb}
\usepackage{microtype}
\usepackage{relsize}
\usepackage{amsmath}

\usepackage{hyperref}

\usepackage[numbers,sort&compress]{natbib}

\DeclareSymbolFont{SY}{U}{psy}{m}{n}
\DeclareMathSymbol{\emptyset}{\mathord}{SY}{'306}

\theoremstyle{plain}

\newtheorem{thm}{Theorem}[section]

\newtheorem{lemma}[thm]{Lemma}

\newtheorem{definition}[thm]{Definition}
\newtheorem{theorem}[thm]{Theorem}

\theoremstyle{definition}

\numberwithin{equation}{section}

\begin{document}
\title[Cyclicity of Cowen-Douglas tuples]{Cyclicity of Cowen-Douglas tuples}

\author{Jing Xu} \author{Shanshan Ji$^{*}$}\thanks{*Corresponding author} \author{Yufang Xie} \author{Kui Ji}

\curraddr[Jing Xu]{School of Mathematics and Science, Hebei GEO University, Shijiazhuang 050031, China}
\curraddr[Kui Ji, Shanshan Ji, and Yufang Xie]{School of Mathematics, Hebei Normal University, Shijiazhuang, Hebei 050016, China}

\email[J. Xu]{xujingmath@outlook.com}
\email[S. Ji]{jishanshan15@outlook.com}
\email[Y. Xie]{xieyufangmath@outlook.com}
\email[K. Ji]{jikui@hebtu.edu.cn}

\thanks{The authors were supported by the National Natural Science Foundation of China, Grant No. 12371129 and 12471123.}

\subjclass[2000]{Primary 47A45; Secondary 46E22, 46J20}

\keywords{Dirichlet Space, Model Theorem, Multipliers, Weights, Backward Shift Operator, Similarity}

\begin{abstract}
The study of Cowen-Douglas operators involves not only operator-theoretic tools but also complex geometry on holomorphic vector bundles. By leveraging the properties of holomorphic vector bundles, this paper investigates the cyclicity of Cowen-Douglas tuples and demonstrates conclusively that every such tuple is cyclic.
\end{abstract}

\maketitle

\section{Introduction}
The shift operator on Hilbert space is well-known. It was first pointed out by Rota \cite{Rota1960}, who regarded such operators as "universal"; this perspective was further elaborated by Foia\c{s} \cite{Foias1963}. Moreover, it is well-established that the backward unilateral shift operator with finite multiplicity is cyclic, as demonstrated by Halmos in his work on cyclic vectors \cite{Halmos1982}. As a natural generalization of the backward shift operator, Cowen-Douglas operators possess rich geometric properties. This class of operators was first introduced by Cowen and Douglas in \cite{CD1}. Later, Lin \cite{lin1988} showed that every Cowen-Douglas operator is cyclic. A natural question is whether the Cowen-Douglas tuples are cyclic. In this paper, we will given an affirmation answer.

Throughout the notes, we denote by $\Omega$ a domain (an open and connected set) in $\mathbb{C}^{m}$, and $\mathbb{Z}_{+}^{m}$ the set of $m$-tuples of nonnegative integers. For $w=(w_{1},\cdots,w_{m})\in\Omega$ and $\alpha=(\alpha_{1},\cdots,\alpha_{m}),\beta=(\beta_{1},\cdots,\beta_{m})\in\mathbb{Z}_{+}^{m}$, we set, as usual,
$$w^{\alpha}=w_{1}^{\alpha_{1}}\cdots w_{m}^{\alpha_{m}},~\alpha+\beta=(\alpha_{1}+\beta_{1},\cdots,\alpha_{m}+\beta_{m}),~\vert \alpha\vert=\alpha_{1}+\cdots+\alpha_{m},~\text{and}~\alpha!=\alpha_{1}!\cdots \alpha_{m}!.$$
Without causing confusion, denote
$$0=(0,\cdots,0),\quad \epsilon=(1,\cdots,1),\quad\text{and} \quad \epsilon_{i}=(0,\cdots,1,\cdots,0)\in\mathbb{Z}_{+}^{m}$$ with 1 on the $i$th position.

Let $\mathcal{H}$ be a complex separable Hilbert space, and let $\mathcal{L}(\mathcal{H})^{m}$ denote the space of all commuting $m$-tuples $\mathbf{T}=(T_{1},\cdots,T_{m})$ of bounded linear operators on $\mathcal{H}$. A $m$-tuple $\mathbf{T}=(T_{1},\cdots,T_{m})$ is said to have a \textit{cyclic vector} $f$ if $$\mathop{\text{span}}\limits_{\alpha\in\mathbf{Z}_{+}^{m}}\{\mathbf{T}^{\alpha}f\}=\mathcal{H},$$
where span\{\quad\} denotes the closed linear span.
Equivalently, $f$ is a cyclic vector for $\mathbf{T}$ if the set of all vectors of the form $P(\mathbf{T})f$, where $P$ ranges over all polynomials, is dense in $\mathcal{H}.$  For convenience, we say that a commuting $m$-tuple is a \textit{cyclic $m$-tuple} if it has a cyclic vector. We associate with the $m$-tuple $\mathbf{T}$ a bounded linear transformation
$$\mathscr{D}_{\mathbf{T}}: \mathcal{H}\longrightarrow\mathcal{H}\oplus\cdots\oplus\mathcal{H}\quad (m\text{~copies}),$$ defined by $\mathscr{D}_{\mathbf{T}}h=(T_{1}h,\cdots,T_{m}h)$ for $h \in \mathcal{H}$. For $\mathbf{T}-w:=(T_{1}-w_{1},\cdots,T_{m}-w_{m})$, it is straightforward to verify that $\ker\mathscr{D}_{\mathbf{T}-w}=\bigcap \limits_{i=1}^{m}\ker(T_{i}-w_{i}).$

\begin{definition}\cite{CD2}
For positive integer $n \in \mathbb{N}$, the Cowen-Douglas tuple $\mathbf{\mathcal{B}}_{n}^{m}(\Omega)$ consists of commuting $m$-tuples $\mathbf{T}=(T_{1},\cdots,T_{m}) \in \mathcal{L}(\mathcal{H})^m$ satisfying the following conditions:
\begin{itemize}
  \item [(i)]$\text{ran}~\mathscr{D}_{\mathbf{T}-w}$ is closed for all $w \in \Omega$;
  \item [(ii)]$\dim \ker\mathscr{D}_{\mathbf{T}-w}=n$ for all $w \in \Omega$; and
  \item [(iii)] $\mathop{\text{span}}\limits_{w \in \Omega}~\ker\mathscr{D}_{\mathbf{T}-w}=\mathcal{H}$.
\end{itemize}
\end{definition}

The set of all $n$-dimensional subspaces of ${\mathcal H}$, called the Grassmannian, is denoted by
$\mbox{Gr}(n,{\mathcal H})$. A map $E$ from $\Omega$ to $\mbox{Gr}(n,{\mathcal H})$ is called a holomorphic vector bundle if, for any point $z_{0}\in\Omega$, there exists a neighborhood $\Delta$ of $z_{0}$ and $n$ holomorphic $\mathcal H$-valued functions $\{\gamma_{1},\cdots,\gamma_{n}\}$, called a holomorphic frame, defined on $\Delta$ such that $$E(z)=\text{span}\{\gamma_{1}(z),\cdots,\gamma_{n}(z)\}\quad \text{for~every}~z\in\Delta.$$
For any $m$-tuple $\mathbf{T}=(T_{1},\cdots,T_{m})\in\mathbf{\mathcal{B}}_{n}^{m}(\Omega),$ Cowen and Douglas proved in \cite{CD1,CD2} that there exists an associated Hermitian holomorphic vector bundle $E_{\mathbf{T}}$ over $\Omega$ of rank $n$. This bundle is defined as
$$E_\mathbf{T}=\{(w, x)\in \Omega\times \mathcal
H: x \in \ker\mathscr{D}_{\mathbf{T}-w}\},\quad \pi(w,x)=w,$$
where $\pi$ denotes the holomorphic map onto $\Omega.$
Moreover, For any holomorphic frame $\{\gamma_{1},\cdots,\gamma_{n}\}$ of $E_\mathbf{T}$, we have
\begin{equation}\label{eq3}
\mathcal{H}=\mathop{\text{span}}\limits_{z\in\Omega}\{\gamma_{i}(z): 1\leq i\leq n\}.
\end{equation}

Recall that a holomorphic cross-section of the Hermitian holomorphic vector bundle $E_{\mathbf{T}}$ is a holomorphic function $\gamma:\Omega\rightarrow\mathcal{H}$ such that, for every $z\in\Omega$, the vector $\gamma(z)$ belongs to the fibre of $E_\mathbf{T}$ over $z$. A holomorphic cross-section $\gamma$ is called a \textit{spanning holomorphic cross-section} if
$$\mathop{\text{span}}\limits_{z\in\Omega}\{\gamma(z)\}=\mathcal{H}.$$
When $\mathbf{T}=(T_{1},\cdots,T_{m})\in\mathbf{\mathcal{B}}_{1}^{m}(\Omega)$, it is evident from (\ref{eq3}) that there is a spanning holomorphic cross-section of $E_{\mathbf{T}}$.
More generally, for single Cowen-Douglas operators $T\in\mathbf{\mathcal{B}}_{n}^{1}(\Omega)$ with $n>1$, Zhu \cite{ZKH} shows that the Hermitian holomorphic vector bundle $E_{T}$ possesses a spanning holomorphic cross-section. For Cowen-Douglas tuples $\mathbf{\mathcal{B}}_{n}^{m}(\Omega)$ with $n, m>1$,  Eschmeier and Schmitt \cite{ES2014} obtained the following result.

\begin{theorem}\cite{ES2014}\label{es}
Let $\Omega\subset\mathbb{C}^{m}$ be a domain of holomorphy, $\mathbf{T}=(T_{1},\cdots,T_{m})\in\mathbf{\mathcal{B}}_{n}^{m}(\Omega)$, and let $\{\gamma_{1},\cdots,\gamma_{n}\}$ be a holomorphic frame of $E_\mathbf{T}$. Then there exist holomorphic functions $\phi_{1},\cdots,\phi_{n}$ such that the mapping $\gamma=\phi_{1}\gamma_{1}+\cdots+\phi_{n}\gamma_{n}:\Omega\rightarrow\mathcal{H}$ is a spanning holomorphic cross-section of $E_\mathbf{T}$, and $\gamma(z)\neq0$ for every $z\in\Omega.$
\end{theorem}

\section{Cyclicity of the class $\mathcal{B}_{n}^{m}(\Omega)$}\label{sec2}
This section begins by presenting two lemmas, which are crucial for the subsequent proof that the Cowen-Douglas tuple is cyclic. Next, we show that a commuting $m$-tuple in $\mathbf{\mathcal{B}}_{1}^{m}(\Omega)$ is cyclic. Finally, leveraging the properties of spanning holomorphic cross-sections, we establish that general Cowen-Douglas tuples are cyclic as well.

\begin{lemma}\label{lem1}
If $l$ and $k$ are positive integers, $\eta\in\mathbb{Z}_{+}^{m}$ satisfies  $|\eta|=l$, and
$$\beta=(k+1+l)\epsilon-\eta\in\mathbb{Z}_{+}^{m}.$$
Then for any $\alpha\in\mathbb{Z}_{+}^{m}$ with $|\alpha|=l$ and $\alpha\neq\eta$,
$$(\beta+\alpha)!>(\beta+\eta)!$$
\end{lemma}
\begin{proof}
Setting $\eta=(i_{1},\cdots,i_{m})$, we have
 $$\beta+\alpha=(k+1+l-i_{1}+\alpha_{1},k+1+l-i_{2}+\alpha_{2},\cdots, k+1+l-i_{m}+\alpha_{m})$$
and $$\beta+\eta=(k+1+l,k+1+l,\cdots, k+1+l),$$
it follows that, for any $\alpha\in\mathbb{Z}_{+}^{m}$ with $|\alpha|=l$ and $\alpha\neq\eta$, $(\beta+\alpha)!>(\beta+\eta)!$ if and only if
\begin{equation*}\label{eq2}
(k+1+l-i_{1}+\alpha_{1})!(k+1+l-i_{2}+\alpha_{2})!\cdots(k+1+l-i_{m}+\alpha_{m})!>[(k+1+l)!]^{m}.
\end{equation*}

Since $\alpha\in\mathbb{Z}_{+}^{m}$, $|\eta|=|\alpha|$ and $\eta\neq\alpha$, there exists a positive integer $s,$ $2\leq s\leq m,$ such that
$$\alpha_{j_{1}}\neq i_{j_{1}},\quad \alpha_{j_{2}}\neq i_{j_{2}}, \quad\cdots, \quad\alpha_{j_{s}}\neq i_{j_{s}},$$ where $j_{1},j_{2}, \cdots, j_{s}\in\{1,2,\cdots,m\}$ and $j_{p}\neq j_{q}$ if $p\neq q$. Without loss of generality, assume that
$$\alpha_{j_{1}}>i_{j_{1}},~\cdots,~\alpha_{j_{p}}>i_{j_{p}},~\alpha_{j_{p+1}}<i_{j_{p+1}},~\cdots, ~\alpha_{j_{s}}<i_{j_{s}}.$$
From the condition $|\alpha|=|\eta|$, we conclude that
$$(\alpha_{j_{1}}+\alpha_{j_{2}}+\cdots+\alpha_{j_{p}})-(i_{j_{1}}+i_{j_{2}}+\cdots+i_{j_{p}})
=(i_{j_{p+1}}+i_{j_{p+2}}+\cdots+i_{j_{s}})-(\alpha_{j_{p+1}}+\alpha_{j_{p+2}}+\cdots+\alpha_{j_{s}}),$$
which implies that
\begin{equation*}
\begin{array}{lll}
\frac{(\beta+\alpha)!}{(\beta+\eta)!}&=&\frac{\prod\limits_{t=1}^{m}(k+1+l-i_{t}+\alpha_{t})!}{[(k+1+l)!]^{m}}\\
&=&\frac{\left[\prod\limits_{t=1}^{p}\left(k+1+l+(\alpha_{j_{t}}-i_{j_{t}})\right)!\right]
\left[\prod\limits_{t=p+1}^{s}\left(k+1+l-(i_{j_{t}}-\alpha_{j_{t}})\right)!\right]
}{[(k+1+l)!]^{s}}\\
&=&\frac{\prod\limits_{t=1}^{p}\left[\prod\limits_{q=1}^{\alpha_{j_{t}}-i_{j_{t}}}(k+1+l+q)\right]}
{\prod\limits_{t=p+1}^{s}\left[\prod\limits_{q=1}^{i_{j_{t}}-\alpha_{j_{t}}}\left(k+1+l-(i_{j_{t}}-\alpha_{j_{t}})+q\right)\right]}\\
&>&1,
\end{array}
\end{equation*}
and finally, we conclude that $(\beta+\alpha)!>(\beta+\eta)!$.
\end{proof}

\begin{lemma}\label{cor1}
Let $l$ and $k$ be positive integers, $\eta\in\mathbb{Z}_{+}^{m}$ with $|\eta|=l$, and
$\beta=(k+1+l)\epsilon\in\mathbb{Z}_{+}^{m}$. Then for any $\alpha\in\mathbb{Z}_{+}^{m}$ with $|\alpha|\geq l$ and $\alpha\neq\eta$,
we have $$(\beta+\alpha-\eta)!>\beta!.$$
\end{lemma}
\begin{proof}
From Lemma \ref{lem1}, it follows that for any $\zeta\in\mathbb{Z}_{+}^{m}$ satisfying $|\zeta|=l$ and $\zeta\neq\eta$, we have
\begin{equation}\label{eq4}
(\beta+\zeta-\eta)!>\beta!.\end{equation}
For any $\epsilon_{i}, i=1,2,\cdots, m$, from
$\beta=(k+1+l)\epsilon> \epsilon$, it follows that
\begin{equation}\label{eq5}
\left(\beta+(\eta+\epsilon_{i})-\eta\right)!=(\beta+\epsilon_{i})!>\beta!.\end{equation}
For any $\alpha\in\mathbb{Z}_{+}^{m}$ with $|\alpha|=l+1$, there exists a $i\in\{1,2,\cdots,m\}$ and $\xi\in\mathbb{Z}_{+}^{m}$ with $|\xi|=l$ such that
$$\alpha=\xi+\epsilon_{i}.$$
If $\xi\neq\eta$, then from (\ref{eq4}), it follows that $$(\beta+\alpha-\eta)!=(\beta+\xi+\epsilon_{i}-\eta)!>(\beta+\xi-\eta)!>\beta!$$
If $\xi=\eta$, then from (\ref{eq5}), it follows that
$$(\beta+\alpha-\eta)!=(\beta+\xi+\epsilon_{i}-\eta)!=(\beta+\xi)!>\beta!.$$
Thus, for any $\alpha\in\mathbb{Z}_{+}^{m}$ satisfying $|\alpha|=l+1$, we have $(\beta+\alpha-\eta)!>\beta!$. Similarly, we can also conclude that
$(\beta+\alpha-\eta)!>\beta!$ holds for any $\alpha\in\mathbb{Z}_{+}^{m}$ and $|\alpha|>l+1$.
Thus, the proof is complete.
\end{proof}

With the preceding lemma established, we are now ready to prove the main theorem. Let $\mathbf{T}=(T_{1},\cdots,T_{m})\in\mathbf{\mathcal{B}}_{1}^{m}(\Omega)$, and let $\gamma$ be a holomorphic frame of $E_{\mathbf{T}}$.
Assume $0\in\Omega$, and write $\gamma(z)=\sum\limits_{\alpha\in\mathbb{Z}_{+}^{m}}a_{\alpha}z^{\alpha}$, where $a_{\alpha}\in\mathcal{H}$.
By (\ref{eq3}), we have
\begin{equation}\label{eq1}
\mathcal{H}=\mathop{\text{span}}\limits_{z\in\Omega}\{\gamma(z)\}=\mathop{\text{span}}\limits_{\alpha\in\mathbb{Z}_{+}^{m}}\{a_{\alpha}\}.
\end{equation}
Notes that $\mathbf{T}^{\beta}\gamma(z)=z^{\beta}\gamma(z)=\sum\limits_{\alpha\in\mathbb{Z}_{+}^{m}}a_{\alpha}z^{\alpha+\beta}
=\sum\limits_{\alpha\geq\beta}a_{\alpha-\beta}z^{\alpha}$, and that $\mathbf{T}^{\beta}\gamma(z)=\sum\limits_{\alpha\in\mathbb{Z}_{+}^{m}}\mathbf{T}^{\beta}a_{\alpha}z^{\alpha}$.
This implies that $\mathbf{T}^{\beta}a_{\alpha}=a_{\alpha-\beta}$ for any $\alpha\geq\beta$; otherwise, $\mathbf{T}^{\beta}a_{\alpha}=0$.

\begin{thm}\label{thm2}
If $\mathbf{T}=(T_{1},\cdots,T_{m})\in\mathbf{\mathcal{B}}_{1}^{m}(\Omega)$, then $\mathbf{T}$ is a cyclic $m$-tuple.
\end{thm}
\begin{proof}
Without loss of generality, we assume $0\in\Omega$ and $\mathbf{T}=(T_{1},\cdots,T_{m}) \in \mathcal{L}(\mathcal{H})^m$.
Let $\gamma$ be a holomorphic frame of $E_{\mathbf{T}}$ and $$\gamma(z)=\sum\limits_{\alpha\in\mathbb{Z}_{+}^{m}}a_{\alpha}z^{\alpha},\quad\text{where}\quad a_{\alpha}\in\mathcal{H}.$$
Since $\Omega$ is a bounded domain, there exist $\delta=(\delta_{1},\cdots,\delta_{m})>0$ and  a constant $M>0$ such that $|z_{i}|\leq \delta_{i}$ for $1\leq i\leq m,$ and $\Vert a_{\alpha}\Vert\leq \frac{M}{\delta^{\alpha}}$ for any $z=(z_{1},\cdots, z_{m})\in\Omega$ and $\alpha\in\mathbb{Z}_{+}^{m}$.
Setting $$f=\sum\limits_{\alpha\in\mathbb{Z}_{+}^{m}}\xi_{\alpha}a_{\alpha},$$ where $\xi_{\alpha}=\frac{1}{\left(|\alpha|!\right)^{\alpha!}},$ we have
$$\Vert f\Vert\leq\sum\limits_{\alpha\in\mathbb{Z}_{+}^{m}}\xi_{\alpha}\Vert a_{\alpha}\Vert\leq M \sum\limits_{\alpha\in\mathbb{Z}_{+}^{m}}\frac{\delta^{-\alpha}}{\left(|\alpha|!\right)^{\alpha!}}\leq
M \sum\limits_{\alpha\in\mathbb{Z}_{+}^{m}}\frac{\delta^{-\alpha}}{\alpha!}=M e^{\frac{1}{\delta_{1}}+\cdots+\frac{1}{\delta_{m}}}<\infty,$$
and thus $f\in\mathcal{H}.$

Now, we prove that $f$ is the cyclic vector of $\mathbf{T}$, that is, $\mathop{\text{span}}\limits_{\alpha\in\mathbb{Z}_{+}^{m}}\{\mathbf{T}^{\alpha}f\}=\mathcal{H}.$
From $$\mathcal{H}=\mathop{\text{span}}\limits_{z\in\Omega}\{\gamma(z)\}=\mathop{\text{span}}\limits_{\alpha\in\mathbb{Z}_{+}^{m}}\{a_{\alpha}\},$$
it suffices to prove that $\{a_{\alpha}:\alpha\in\mathbb{Z}_{+}^{m}\}\subset\mathop{\text{span}}\limits_{\alpha\in\mathbb{Z}_{+}^{m}}\{\mathbf{T}^{\alpha}f\}$.

For any positive integer $k$ and any $\alpha\in\mathbb{Z}_{+}^{m}\backslash\{0\}$, we have that $((k+1)\epsilon+\alpha)!>((k+1)\epsilon)!$, and then
\begin{eqnarray*}
&&\lim\limits_{k\rightarrow\infty}\left\Vert\frac{1}{\xi_{(k+1)\epsilon}}\mathbf{T}^{(k+1)\epsilon}f-a_{0}\right\Vert\\
&=&\lim\limits_{k\rightarrow\infty}
\left\Vert\sum\limits_{\alpha\in\mathbb{Z}_{+}^{m}}\frac{\xi_{(k+1)\epsilon+\alpha}}{\xi_{(k+1)\epsilon}}a_{\alpha}-a_{0}\right\Vert\\
&=&\lim\limits_{k\rightarrow\infty}
\left\Vert\sum\limits_{\substack{\alpha\in\mathbf{Z}_{+}^{m}  \\ \alpha\neq0 }}\frac{\xi_{(k+1)\epsilon+\alpha}}{\xi_{(k+1)\epsilon}}a_{\alpha}\right\Vert\\
&\leq& M \lim\limits_{k\rightarrow\infty}\sum\limits_{\alpha\in\mathbb{Z}_{+}^{m}\atop \alpha>0}\frac{(|(k+1)\epsilon|!)^{((k+1)\epsilon)!}}{(|(k+1)\epsilon+\alpha|!)^{((k+1)\epsilon+\alpha)!}}\delta^{-\alpha}\\
&=&M \lim\limits_{k\rightarrow\infty}\sum\limits_{\alpha\in\mathbb{Z}_{+}^{m}\atop \alpha>0}\left(\frac{|(k+1)\epsilon|!}{|(k+1)\epsilon+\alpha|!}\right)^{((k+1)\epsilon)!}\frac{\delta^{-\alpha}}{(|(k+1)\epsilon+\alpha|!)^{((k+1)\epsilon+\alpha)!-((k+1)\epsilon)!}}\\
&\leq&\lim\limits_{k\rightarrow\infty} \frac{M}{(k+1)!} \sum\limits_{\alpha\in\mathbb{Z}_{+}^{m}\atop \alpha>0}\frac{\delta^{-\alpha}}{\alpha!}\\
&\leq&\lim\limits_{k\rightarrow\infty}\frac{M}{(k+1)!}  e^{\frac{1}{\delta_{1}}+\cdots+\frac{1}{\delta_{m}}}\\
&=&0.
\end{eqnarray*}
It follows that $a_{0}\in\mathop{\text{span}}\limits_{\alpha\in\mathbb{Z}_{+}^{m}}\{\mathbf{T}^{\alpha}f\}$. Without loss of generality,
we can assume that $$\{a_{\alpha}:\alpha\in\mathbb{Z}_{+}^{m}~\text{and}~0\leq|\alpha|\leq l\}\subset\mathop{\text{span}}\limits_{\alpha\in\mathbb{Z}_{+}^{m}}\{\mathbf{T}^{\alpha}f\}$$ for some non-negative integer $l$.
Then for any $\alpha\in\mathbb{Z}_{+}^{m}$ and $|\alpha|=l+1$, by Lemma \ref{lem1} and Lemma \ref{cor1}, we obtain that
\begin{align*}
&\quad\,\,\,\lim\limits_{k\rightarrow\infty}\left\Vert\frac{1}{\xi_{(k+l+2)\epsilon}}\mathbf{T}^{(k+l+2)\epsilon-\alpha}f-\sum\limits_{\beta\in\mathbb{Z}_{+}^{m}\atop |\beta|\leq l}\frac{\xi_{\beta+(k+l+2)\epsilon-\alpha}}{\xi_{(k+l+2)\epsilon}}a_{\beta}-a_{\alpha}\right\Vert\\
&=\lim\limits_{k\rightarrow\infty}
\left\Vert\sum\limits_{\eta\in\mathbb{Z}_{+}^{m}}\frac{\xi_{\eta+(k+l+2)\epsilon-\alpha}}{\xi_{(k+l+2)\epsilon}}a_{\eta}-
\sum\limits_{\beta\in\mathbb{Z}_{+}^{m}\atop |\beta|\leq l}\frac{\xi_{\beta+(k+l+2)\epsilon-\alpha}}{\xi_{(k+l+2)\epsilon}}a_{\beta}-a_{\alpha}\right\Vert \displaybreak[1] \\
&=\lim\limits_{k\rightarrow\infty}
\left\Vert\sum\limits_{\substack{\eta\in\mathbb{Z}_{+}^{m}  \\ |\eta|\geq l+1, \eta\neq\alpha }}\frac{\xi_{\eta+(k+l+2)\epsilon-\alpha}}{\xi_{(k+l+2)\epsilon}}a_{\eta}\right\Vert\\
&\leq M \lim\limits_{k\rightarrow\infty}\sum\limits_{\substack{\eta\in\mathbb{Z}_{+}^{m}  \\ |\eta|\geq l+1, \eta\neq\alpha }}\frac{(|(k+l+2)\epsilon|!)^{((k+l+2)\epsilon)!}}{(|(k+l+2)\epsilon+\eta-\alpha|!)^{((k+l+2)\epsilon+\eta-\alpha)!}}\delta^{-\eta}\\ 
&=M \lim\limits_{k\rightarrow\infty}\sum\limits_{\substack{\eta\in\mathbb{Z}_{+}^{m}  \\ |\eta|\geq l+1, \eta\neq\alpha }}\left(\frac{|(k+l+2)\epsilon|!}{|(k+l+2)\epsilon+\eta-\alpha|!}\right)^{((k+l+2)\epsilon)!}\frac{\delta^{-\eta}}{(|(k+l+2)\epsilon+\eta-\alpha|!)
^{((k+l+2)\epsilon+\eta-\alpha)!-((k+l+2)\epsilon)!}}\\
&\leq\lim\limits_{k\rightarrow\infty} \frac{M}{(k+1)!} \sum\limits_{\substack{\eta\in\mathbb{Z}_{+}^{m}  \\ |\eta|\geq l+1, \eta\neq\alpha }}\frac{\delta^{-\eta}}{\eta!}\\
&\leq\lim\limits_{k\rightarrow\infty}\frac{M}{(k+1)!}  e^{\frac{1}{\delta_{1}}+\cdots+\frac{1}{\delta_{m}}}\\
&=0.
\end{align*}
Thus, $\{a_{\alpha}:\alpha\in\mathbb{Z}_{+}^{m}~\text{and}~|\alpha|=l+1\}\subset\mathop{\text{span}}\limits_{\alpha\in\mathbb{Z}_{+}^{m}}\{\mathbf{T}^{\alpha}f\}$.
Therefore, $$\{a_{\alpha}:\alpha\in\mathbb{Z}_{+}^{m}\}\subset\mathop{\text{span}}\limits_{\alpha\in\mathbb{Z}_{+}^{m}}\{\mathbf{T}^{\alpha}f\},$$ and hence $f$ is the cyclic vector of $\mathbf{T}$.
\end{proof}

The above theorem holds for a commuting $m$-tuple $\mathbf{T}=(T_{1},\cdots,T_{m})$ in $\mathbf{\mathcal{B}}_{1}^{m}(\Omega)$. Building on this, we extend the proof to arbitrary commuting $m$-tuple $\mathbf{T}=(T_{1},\cdots,T_{m})$ in $\mathbf{\mathcal{B}}_{n}^{m}(\Omega)$, where $n\geq 1.$



\begin{thm}
If $\mathbf{T}=(T_{1},\cdots,T_{m})\in\mathbf{\mathcal{B}}_{n}^{m}(\Omega)$, then $\mathbf{T}$ is a cyclic operator tuple.
\end{thm}
\begin{proof}
Without loss of generality, we assume $0\in\Omega$ and $\mathbf{T}=(T_{1},\cdots,T_{m}) \in \mathcal{L}(\mathcal{H})^m$.
Next, we assume that $\Omega$ is a domain of holomorphy. Otherwise, since $0\in\Omega$, there exists a domain of holomorphy $\Omega_{0}$ that contains the origin. From $\mathbf{\mathcal{B}}_{n}^{m}(\Omega)\subset\mathbf{\mathcal{B}}_{n}^{m}(\Omega_{0})$ (see Lemma 4.4 in \cite{ES2014}), it follows that
$\mathbf{T}=(T_{1},\cdots,T_{m})\in\mathbf{\mathcal{B}}_{n}^{m}(\Omega_{0}).$

Let $\{\gamma_{1},\cdots, \gamma_{n}\}$ be a holomorphic frame of $E_{\mathbf{T}}$. From Theorem \ref{es}, there exist holomorphic functions $\phi_{1},\cdots, \phi_{n}\in\mathcal{O}(\Omega)$ such that $$\gamma=\phi_{1}\gamma_{1}+\cdots+\phi_{n}\gamma_{n}:\Omega\rightarrow\mathcal{H}$$ is
a spanning holomorphic cross-section for $\mathbf{T}$, that is, $$\mathcal{H}=\mathop{\text{span}}\limits_{z\in\Omega}\{\gamma(z)\}.$$ Since $0\in\Omega,$ we write $\gamma(z)=\sum\limits_{\alpha\in\mathbb{Z}_{+}^{m}}a_{\alpha}z^{\alpha}$, where $a_{\alpha}\in\mathcal{H}$, then
$$\mathcal{H}=\mathop{\text{span}}\limits_{z\in\Omega}\{\gamma(z)\}=\mathop{\text{span}}\limits_{\alpha\in\mathbb{Z}_{+}^{m}}\{a_{\alpha}\}.$$
Letting $f=\sum\limits_{\alpha\in\mathbb{Z}_{+}^{m}}\xi_{\alpha}a_{\alpha}$, where $\xi_{\alpha}=\frac{1}{\left(|\alpha|!\right)^{\alpha!}},$
and proceeding similarly to the proof of Theorem \ref{thm2}, we conclude that $\mathbf{T}$ is a cyclic $m$-tuple.
\end{proof}

\end{document}